\documentclass[10pt,a4paper]{article}

\normalfont
\usepackage{amsmath,amsthm,amssymb,enumerate}
\usepackage{graphicx,  float, afterpage,array,multicol,eucal, url,color}

 \usepackage{amscd}
\usepackage{tikz}
\usetikzlibrary{arrows,shapes}

\font\tenmsb=msbm10 \font\sevenmsb=msbm7 \font\fivemsb=msbm5
\newfam\msbfam
\textfont\msbfam=\tenmsb \scriptfont\msbfam=\sevenmsb
\scriptscriptfont\msbfam=\fivemsb

\font\teneufm=eufm10 \font\seveneufm=eufm7 \font\fiveeufm=eufm5
\newfam\eufmfam
\textfont\eufmfam=\teneufm \scriptfont\eufmfam=\seveneufm
\scriptscriptfont\eufmfam=\fiveeufm

\def\co#1{
}

\renewcommand{\epsilon}{\varepsilon}
\renewcommand{\setminus}{\smallsetminus}
\renewcommand{\emptyset}{\varnothing}

\newtheorem{theorem}{Theorem}[section]
\newtheorem{proposition}[theorem]{Proposition}
\newtheorem{corollary}[theorem]{Corollary}
\newtheorem{lemma}[theorem]{Lemma}

\newtheorem{question}[theorem]{Question}

\newtheorem{example}[theorem]{Example}
\newtheorem{examples}[theorem]{Examples}

\newtheorem{definition}[theorem]{Definition}

\newtheorem{remark}[theorem]{Remark}

\newcommand{\pd}{\operatorname{pd}}
\newcommand{\cd}{\operatorname{cd}}
\newcommand{\vcd}{\operatorname{vcd}}

\renewcommand{\AA}{\mathcal A}
\newcommand{\NN}{\mathcal N}

\newcommand{\FF}{\mathcal F}

\newcommand{\HH}{\mathcal H}

\newcommand{\Z}{\mathbb Z}

\newcommand{\R}{\mathbb R}

\newcommand{\fpinfty}{{\FP}_{\infty}}

\newcommand{\F}{\operatorname{F}}
\newcommand{\UF}{\underline{\operatorname{F}}}
\newcommand{\FP}{\operatorname{FP}}

\newcommand{\UFP}{\underline{\operatorname{FP}}}

\newcommand{\cohom}[3]{H^{{\raise1pt\hbox{$\scriptstyle#1$}}}(#2\>\!,#3)}
\newcommand{\tatecohom}[3]%
  {\widehat H^{{\raise1pt\hbox{$\scriptstyle#1$}}}(#2\>\!,#3)}

\newcommand{\Cohom}[3]%
  {H^{{\raise1pt\hbox{$\scriptstyle#1$}}}\big(#2\>\!,#3\big)}
\newcommand{\Tatecohom}[3]%
  {\widehat H^{{\raise1pt\hbox{$\scriptstyle#1$}}}\big(#2\>\!,#3\big)}

\newcommand{\homol}[3]{H_{{\lower1pt\hbox{$\scriptstyle#1$}}}(#2\>\!,#3)}
\newcommand{\homolog}[2]{H_{{\lower1pt\hbox{$\scriptstyle#1$}}}(#2)}

\newcommand{\egamma}{\underline{\operatorname{E}}\Gamma}

\newcommand{\OHGamma}{\mathcal O_{\mathcal H}\Gamma}

\begin{document}

\title{Euler classes and Bredon cohomology for groups with restricted families of finite subgroups}

\author{Conchita Mart\'{\i}nez-P\'{e}rez\\
Instituto Universitario de Matem\'aticas y Aplicaciones,\\
Departamento de Matem\'{a}ticas,\\
Universidad de Zaragoza, \\
 50009 Zaragoza \\
 Phone: +34976761000 Ext: 3243 \\
 Fax: +34976761338 \\
 email: conmar@unizar.es \\}
       

\maketitle

\noindent{\bf keywords: }Bredon cohomology, 
classifying space for proper actions, poset of finite subgroups, Euler classes.


\begin{abstract} We explore some of the special features with respect to Bredon cohomology for groups whose finite subgroups are all either nilpotent or $p$-groups or cyclic $p$-groups. We get some results on dimensions and also a formula for the equivariant Euler class for certain groups. We consider the generalization for Bredon cohomology of the properties of being duality or Poincar\'e duality  and study their behavior under $p$-power index extensions with coefficients in a field of characteristic $p$.\end{abstract}

\section{Introduction}
\noindent Posets of subgroups  are relevant objects in finite group theory.  However, almost nothing is known about these kind of posets for arbitrary groups. A few of the known properties of the finite case translate without problems to  arbitrary groups, however most of the known results  have no obvious translation. 

Classifying spaces for proper actions $\egamma$ (see Section \ref{first} for a definition) yield a
 motivation to study the poset of non-trivial finite subgroups $\FF_1$ in an arbitrary group $\Gamma$. Originally due to its appearance in the famous Baum-Connes conjecture, these spaces have received a lot of attention recently. Bredon cohomology is a cohomology theory that can be defined using the space $\egamma$ exactly as ordinary cohomology can be defined using the ordinary classifying space $\text{E}\Gamma$. We also use the term to refer to the study of the group properties which are algebraic counterparts of  properties  of $\egamma$ (such as minimal dimension, various finiteness conditions, etc).  The connexion between $\FF_1$ and $\egamma$ was first discovered in \cite{connkoz} by Connolly and Kozniewsky: Let $\egamma_s$ be the subcomplex of $\egamma$ formed by the cells on which $\Gamma$ does not act freely. If $|\FF_1|$ is the geometric realization of $\FF_1$ and $C$ is the cone construction, then there is an equivariant map of pairs
  $$(\egamma,\egamma_s)\to (C|\FF_1|,|\FF_1|)$$ 
 which induces a non-equivariant homotopy equivalence (an algebraic version of this fact valid for arbitrary families of subgroups can be found in \cite{conch2}).

This paper  is in some sense the second part of \cite{conch2}. There we  concentrated in the case of virtually solvable groups to study certain aspect of their posets of finite subgroups and to derive consequences in Bredon cohomology. Here we adopt a different strategy and impose restrictions to the possible finite subgroups of the groups we are considering. We begin by briefly studying groups having all of their finite subgroups nilpotent and then we will restrict ourselves to the cases when all the finite subgroups are either $p$-groups or $p$ or cyclic $p$-groups where in each case $p$ denotes a fixed prime. Our aim is also to gain a better understanding of the chain complexes associated to the posets.
We get some results concerning dimensions, for example from Theorem \ref{dimensionnilpotent}, Lemma \ref{cyclic} and Example \ref{sharpbound} we get (here $R$ is an abelian coefficient ring, see Section \ref{first} for the rest of notations)
\bigskip

\noindent{\bf Theorem A:} {\sl  Let $\Gamma$ be a group having all its finite subgroups nilpotent of bounded order. Then
$${\underline\cd}_R\Gamma\leq\text{max}_{H\in\FF}\big[\pd_{R\Gamma} B(RWH)+r(WH)\big].$$
In the particular case when all the finite subgroups are
cyclic $p$-groups this yields
$$\pd_{R\Gamma} B(R\Gamma)\leq\underline{\cd}_R\Gamma\leq 1+\pd_{R\Gamma }B(R\Gamma)$$
and there are examples where either side of the inequality is sharp.}

\bigskip

In Section \ref{euler}, we consider equivariant Euler classes for groups with the same restrictions on their families of finite subgroups. The equivariant Euler class of a proper cocompact $\Gamma$-space is an invariant first defined by L\"uck in \cite{LuckL2} Definition 6.8.4. The equivariant Euler class of the classifying space for proper actions, denoted $\chi^\Gamma(\egamma)$, enables one to compute several Euler characteristics for $\Gamma$ (see \cite{conch2}). In \cite{conch2} we computed a formula for $\chi^\Gamma(\egamma)$ in the case when $\Gamma$ is elementary amenable of type $\FP_\infty$. In Theorem \ref{formulaeuler} below we prove (now, we refer to Section \ref{euler} for notation)

\bigskip

\noindent{\bf Theorem B:} {\sl  Assume that $\Gamma=K\ltimes G$ with $G$ torsion free and $K$ a finite $p$-group. 
Assume moreover that for any $G\leq S\leq\Gamma$ all those finite subgroups $H\leq\Gamma$ with $HG=S$ are $\Gamma$-conjugated. Then the coefficient of $[\Gamma/1]$ in $\chi(\egamma)^\Gamma$ is
$$\sum_{H\in\AA(K)/K}{(-1)^{\text{lg}_p|H|}p^{\Big({\text{lg}_p|H|\atop 2}\Big)}\over|N_K(H)|}\chi(C_G(H))$$
where we adopt the convention that $\Big({0\atop 2}\Big)=\Big({1\atop 2}\Big)=1$
}

\bigskip

The hypothesis of this theorem hold true for certain finite extensions of right-angled Artin groups and in Section \ref{euler} we use the formula above to compute explicitly the Euler classes in some examples.

In the last two Sections we concentrate in Bredon cohomology with coefficients in a field $F$ of prime characteristic $p$. Hamilton has shown  that if $\Gamma$ is a split extension of a torsion free group $G$ by a finite $p$-group, then $\Gamma$ is of type $\UFP_\infty$ over $F$  (see Definition \ref{bredonfpinfty}) if and only if $G$ is $\FP_\infty$ over $F$. Davis and Leary have generalized the notion of Poincar\'e duality to Bredon cohomology, the notion of ordinary duality can be generalized in the same spirit.  We will prove that the corresponding assertion to Hamilton's also holds true with respect to being \lq\lq Bredon Poincar\'e duality over $F$" but that  the same is not true for Bredon duality. In other words, our Theorem \ref{poincaredualityext} and the example of Section \ref{dualityext} yield:

\bigskip

\noindent{\bf Theorem C:} {\sl Let $\Gamma=K\ltimes G$ with $G$ torsion free and $K$ a $p$ group. Then
\begin{itemize}
\item[i)] $G$ is Poincar\'e duality over $F$ if and only if $\Gamma$ is Bredon Poincar\'e duality over $F$. In this case, $\underline{\cd}_F\Gamma=\cd_FG$.

\item[ii)] For $p=2$, there are examples where $G$ is duality over $F$ but $\Gamma$ is not Bredon duality over $F$.
\end{itemize}}
\bigskip
In particular (see Definitions \ref{bredonfpinfty} and \ref{bredonPD} below) this means that with the same notation as in the Theorem, for any $H\leq K$, the group $C_G(H)$ is Poincar\'e duality over $F$.
This generalizes Corollary 2.1 in \cite{farrelllafont} where the same is proven under the extra hypothesis that $G$ is $\delta$-hyperbolic Poincar\'e duality over $\Z$ and leads to the following

 \bigskip

\noindent{\bf Question D:}  For which coefficient rings and classes of torsion free groups $G$ is it true that any finite index extension $\Gamma$ of $G$ is:
\begin{itemize}
\item[i)]  $\UFP_\infty$ if $G$ is $\FP_\infty$?

\item[ii)] Bredon duality if $G$ is duality?

\item[iii)] Bredon Poincar\'e duality if $G$ is Poincar\'e duality?

\item[iv)] Has $\underline{\cd}\Gamma=\vcd\Gamma$?
 \end{itemize}
\bigskip

At the end of Section \ref{poincareduality} we consider some examples related to Question D. All the examples constructed in this paper are finite index extensions of either Bestvina-Brady groups or right-angled Artin groups. The main reason for that is that for these kind of groups we have a number of results available which provide us with a good understanding of their ordinary and Bredon cohomological properties (\cite{learynucinkis}, \cite{IanMuge}, \cite{charneydavis}). For the relevant definitions the reader is referred to any of those papers.

All throughout the paper, $R$ is an abelian coefficient ring and whenever we do not mention it it is assumed that coefficients are taken in $\Z$.

\section{Basics on Bredon cohomology and first results}\label{first}

\noindent Let $\HH$ be a family of subgroups of a group $\Gamma$, i.e.,  a set of subgroups which is closed under $\Gamma$-conjugation and  taking subgroups. 
A $\Gamma$-$CW$-complex $X$ is said to be a classifying space for $\Gamma$ with respect to $\HH$ if $X^H$ is contractible whenever $H\in\HH$ and empty otherwise. In this case, $X$ is denoted $\text{E}_\HH\Gamma$.
These kind of spaces are known to exist and to be unique up to $\Gamma$-homotopy equivalence (\cite{tomdieck} I 6). 

For each $H\in\HH$, the augmented chain complex of $X^H$ is an exact chain complex of $WH$-modules where $WH:=N_\Gamma(H)/H$ is the Weyl group. This means that $X$ yields a family of resolutions of the trivial module, one for each  subgroup $H\in\HH$. Consider now the category $\OHGamma$ (see \cite{mislin}). Its objects are the  transitive $\Gamma$-sets with stabilizers in $\HH$, (i.e. $\Gamma$-sets of the form $\Gamma/H$ with $H\in\HH$) and its morphisms are $\Gamma$-maps. A Bredon contramodule or just a contramodule is a functor from $\OHGamma$ to the category of $R$-modules. For example, the constant functor $\Gamma/H\mapsto R$ is called the trivial contramodule. Bredon contramodules together with natural transformations between them form a new category which has products and coproducts. Let $H,K\leq\Gamma$ and denote by  $ R[\Gamma/H,\Gamma/K]$ the free $R$-module generated by the $\Gamma$-maps from $\Gamma/H$ to $\Gamma/K$.
Via Yoneda's lemma one easily sees that  $\Gamma/H\mapsto R[\Gamma/H,\Gamma/K]$ where $K\in\HH$ is a projective Bredon contramodule.
 
From this it is easy to see that the category of contramodules has enough projectives so we may form projective resolutions and define homology and cohomology in the usual way. For example, the Bredon chain complex mapping each $\Gamma/H$ to the chain complex of  $X^H$ is a Bredon projective resolution of the trivial contramodule. As in ordinary cohomology, $\cd_{R\HH}\Gamma$ is the smallest length of a projective resolution of the trivial contramodule. In the particular case when $\HH=\FF$ is the family of the finite subgroups of $\Gamma$ we will write $\underline{\cd}_R\Gamma:=\cd_{R\FF}\Gamma$ and $\egamma:=\text{E}_\FF\Gamma$. As in the ordinary case, except for some small dimension exceptions, $\underline{\cd}_R\Gamma$ equals the minimal dimension of a model for $\egamma.$

We view $\HH$ as a poset and denote by $\HH_\bullet$ the $R\Gamma$-chain complex of the geometrical realization $|\HH|$ of $\HH$. Also, if $H\leq \Gamma$, we use $\HH_H$, or $\HH_H(\Gamma)$ is there is possible confusion about the ambient group to denote  the poset
$$\HH_H(\Gamma)=\HH_H:=\{K\in\HH\mid H<K\}.$$

By \cite{conch2} (Lemma 2.3, Remark 2.4 and Theorem 2.5),
a projective Bredon resolution of the trivial contramodule is determined by
the $WH$-posets $\HH_H$. This is the algebraic version of the homotopy equivalence mentioned in the introduction which goes back to  Connolly and Kozniewski (\cite{connkoz}).

Sometimes, we can consider the smaller poset  $\HH_{1}(WH)$:

\begin{lemma}\label{nilpotent} If $H\in\HH$ has the property that for any $H<K\in\HH$,
$H< N_K(H)$ then there is a $WH$-homotopy equivalence 
$$|{\HH}_1(WH)|\simeq_{WH}|{\HH}_H(\Gamma)|$$
\end{lemma}
\begin{proof} Consider the inclusion which is a $WH$-map of posets:
$$
i:{\HH}_1(WH)\to{\HH}_H(\Gamma).
$$
By \cite{Ben2} Theorem 6.4.2 we have to prove that for any $H\leq S\leq N_\Gamma(H)$ the induced map $i^S:{\HH}_1(WH)^S\to{\HH}_H(\Gamma)^S$ between the fixed points subposets yields an ordinary homotopy equivalence. 
For any $K\in{\HH}_H(\Gamma)^S$
$$\begin{aligned}
i^S/K^S:&=\{H<T\leq N_\Gamma(H)\mid T\leq K\text{ and }S\leq N_\Gamma(T)\}\\
&=\{H<T\leq N_K(H)\mid S\leq N_\Gamma(T)\}
\end{aligned}$$
But since $S\leq N_\Gamma(K)$, we have $S\leq N_\Gamma(N_K(H))$ thus $|i^S/K^S|\simeq\ast$
and by Quillen's Lemma (see for example \cite{Ben2} 6.5.2), the map $i^S$ is a homotopy equivalence.
\end{proof}

The hypothesis of Lemma \ref{nilpotent} does not always hold (see \cite{conch2} Remark 2.8), but note that it does hold if all the groups in $\HH$ happen to be nilpotent.
We denote by $\NN$ and $\AA$ respectively the posets of finite nilpotent subgroups and of finite elementary abelian subgroups of $\Gamma.$
Then $\NN_1$ and $\AA_1$ are the corresponding subposets of non-trivial subgroups.

\begin{lemma}\label{quillen}(Quillen-Thevenaz) Let $\Gamma$ be an arbitrary group. There is a $\Gamma$-homotopy equivalence
$$|\NN_1|\simeq_\Gamma|\AA_1|.$$
\end{lemma}
\begin{proof} This follows from the standard proof for the case of a finite group which can be found for example  in \cite{Ben2} Theorem 6.6.1 and which we sketch below. 
Let $i:\AA_1\to\NN_1$ be the inclusion. Again, we have to check that for any $S\leq\Gamma$ the induced map $i^S:\AA_1^S\to\NN_1^S$ yields an ordinary homotopy equivalence and by Quillen's Lemma it suffices to check that 
 for each $P\leq\Gamma$ nilpotent normalized by $S$, the poset $\AA_1(P)^S$ consisting of those non-trivial elementary abelian subgroups of $P$ which are also normalized by $S$ is contractible. Note that as $P$ is nilpotent, $Z(P)\neq 1$. Let $P_0\leq Z(P)$ the subgroup generated by the elements in $Z(P)$ of prime order, then $P_0$ is normal in $P$ and normalized by $S$ so  the poset $\AA_1(P)^S$ is conically contractible (\cite{Ben2} Definition 6.4.6).
\end{proof}

We have chosen to prove the full versions of Lemmas \ref{nilpotent} and \ref{quillen}, note however that in the sequel we will only need the weaker fact that in either case there is an equivariant map which is an ordinary homotopy equivalence, as this yields an equivariant quasi-isomorphism  between the chain complexes associated to the geometric realizations of each poset.

We let $\Sigma\widetilde C_\bullet$ be the result of augmenting and suspending the chain complex $C_\bullet$. A projective resolution of $C_\bullet$ is a chain complex of projectives $P_\bullet$ together with a $\Gamma$-quasi-isomorphism $P\bullet\to C_\bullet$.
And the projective dimension $\pd_{R\Gamma}C_\bullet$ is, as for single modules, the smallest length of a projective resolution of $C_\bullet$.
Using the Comparison Theorem it is easy to show that quasi-isomorphic chain complexes have the same projective dimension (see Lemma 2.1 of \cite{conch2}).


Here and throughout the paper, for any family of subgroups $\HH$ of a group $\Gamma$ we denote by $\HH/\Gamma$ a set of representatives of the $\Gamma$-conjugacy orbits in $\HH$.
Combining Theorem 2.6 in \cite{conch2}  and Lemmas \ref{nilpotent} and \ref{quillen} we immediately get 

\begin{theorem}\label{dim2} Assume that all the finite subgroups of $\Gamma$ are nilpotent of bounded order. Then
$$\underline{\cd}_R\Gamma=\max_{H\in\FF/\Gamma}\pd_{RWH}\Sigma\widetilde\AA_1(WH)_\bullet.$$
\end{theorem}

Recall that the $R\Gamma$-module of bounded functions (first defined in \cite{krophollertalelli}) is:
$$B(R\Gamma):=\{f:\Gamma\to R\mid f(\Gamma)\text{ is finite}\}.$$ 
One of its main features is that it is projective when restricted to any finite subgroup of $\Gamma$, so it can have finite projective dimension. In fact, its projective dimension is a group invariant such that for virtually torsion free groups $\Gamma$, $\pd_{R\Gamma} B(R\Gamma)=\vcd_R\Gamma$ (see \cite{conch} Lemma 3.9).

\begin{lemma}\label{lowerbound} Let  $A_\bullet$ an $R$-free chain complex of $R\Gamma$-modules with $\pd_{R\Gamma}A_\bullet<\infty$. Then
$$\pd_{R\Gamma} A_\bullet\leq\rm{length}A_\bullet+\pd_{R\Gamma} B(R\Gamma)$$
\end{lemma}
\begin{proof} The result is obvious if $\pd_{R\Gamma} B(R\Gamma)=\infty$, so we may assume it is finite. By the version of  \cite{conch} Lemma 3.4 for chain complexes (which can be proven analogously),
$$\pd_{R\Gamma} A_\bullet=\pd_{R\Gamma}\big[A_\bullet\otimes_R B(R\Gamma)\big].$$

Now, take a projective resolution $P_\bullet\to B(R\Gamma)_\bullet$ where $B(R\Gamma)_\bullet$ is seen as a complex concentrated in degree 0. By the K\"unneth Theorem $A_\bullet\otimes_R P_\bullet$ has the same homology as $A_\bullet\otimes_R B(\Gamma)$, therefore the induced map $A_\bullet\otimes_R P_\bullet\to A_\bullet\otimes_R B(\Gamma)$ is a projective resolution thus
$$\begin{aligned}
\pd_{R\Gamma}[A_\bullet\otimes_R B(R\Gamma)]&=\pd_{R\Gamma}(A_\bullet\otimes_RP_\bullet)\\
&\leq\text{length}(A_\bullet\otimes_R P_\bullet)=\text{length}A_\bullet+\pd_{R\Gamma} B(R\Gamma).\\
\end{aligned}$$
\end{proof}

\begin{theorem}\label{dimensionnilpotent} Assume that all the finite subgroups of $\Gamma$ are nilpotent. Then
$$\pd_{R\Gamma} B(R\Gamma)\leq\underline\cd_R\Gamma\leq\max_{H\in\FF/G}\big[\pd_{RWH}B(RWH)+r(WH)\big]$$
where $r(WH)$ is the biggest rank of a finite elementary abelian subgroup of $WH$.
\end{theorem}
\begin{proof} The lower bound is well known. The upper bound follows from Theorem \ref{dim2} and Lemma \ref{lowerbound}.
\end{proof}

One more consequence of the nilpotency of the finite subgroups is the next result which will be useful below.

\begin{lemma}\label{contractible} Let $\Gamma$ be a group having all its finite subgroups nilpotent. If there are $H<K$ both finite with $|C_\Gamma(H):C_\Gamma(K)|<\infty$, then $|\FF_H|\simeq\ast$.
\end{lemma}
\begin{proof} Recall that by Lemma \ref{nilpotent}, there is a $WH$-homotopy equivalence 
$$|{\FF}_1(WH)|\simeq|{\FF}_H(\Gamma)|.$$
Let $L:=N_K(H)=K\cap N_\Gamma(H)$, observe that as $K$ is nilpotent, we have $H<L$ thus $L/H\in{\FF}_1(WH).$
Obviously, $C_\Gamma(K)\leq C_\Gamma(L)\leq C_\Gamma(H)$ thus the index $|C_\Gamma(H):C_\Gamma(L)|$ is also finite.
And as each centralizer has finite index in the corresponding normalizer, this implies  that also $r:=|N_\Gamma(H):N_\Gamma(L)\cap N_\Gamma(H)|<\infty$. Now, the argument in the proof of \cite{connkoz} Corollary 6.1 shows that there is some $L\leq S\trianglelefteq N_\Gamma(H)$ with $S/H\in{\FF}_1(WH)$ thus ${\FF}_1(WH)$ is conically contractible. For the reader's convenience, we recall briefly that argument here. Let $S$ denote the set of all products of conjugates of elements of $L$ by elements in $N_\Gamma(H)$, then obviously $L\leq S\trianglelefteq N_\Gamma(H)$. We claim that $S$ is finite. To see it, it suffices to observe that whenever $l,t\in L$ and $x,y\in N_\Gamma(H)$, there are some $h\in L$ and $z\in N_\Gamma(H)$ such that $l^xt^y=t^yh^z$ thus every element of $S$ can be written as a product of at most $r$ conjugates of elements of $L$. As $L$ is finite, the claim follows.

\end{proof}

\begin{theorem}\label{condition} Let $\Gamma$ be a group with all its finite subgroups nilpotent of bounded order and with
$\underline{\cd}_R\Gamma<\infty$. Assume that there is an order reversing integer valued function $l:\FF\to\Z$ such that
for each $H\leq \Gamma$ finite $\pd_{RWH}B(WH)\leq l(H)$ and either $l(K)<l(H)$ for any $H<K\in\FF$ or there is some $H<K\in\FF$ with $N_\Gamma(K)$ and $N_\Gamma(H)$ commensurable.
Then $\underline{\cd}_R\Gamma\leq l(1).$\end{theorem}
\begin{proof} Use Theorem 2.12 of \cite{conch2} and Lemma \ref{contractible}.
\end{proof}

\section{The case of cyclic or rank 2 elementary abelian finite $p$-subgroups}\label{smallrank}

\noindent Obviously, the results of the previous Section apply for groups having all their finite subgroups $p$-groups of bounded order. In the particular case when moreover all the finite subgroups are cyclic things work specially well. Note that in this case,
$$\AA_1=\{H\leq \Gamma\mid |H|=p\}.$$
Therefore we have:

\begin{lemma}\label{precyclic} Let $p$ be a prime and $\Gamma$ a group such that all its finite subgroups are cyclic $p$-groups of bounded order. Then
$$\pd_{R\Gamma} B(R\Gamma)\leq\underline{\cd}_R\Gamma\leq\pd_{R\Gamma} B(R\Gamma)+1$$
\end{lemma}
\begin{proof} With the notation of Theorem \ref{dimensionnilpotent}, $r(WH)\leq 1$ for any $H\leq\Gamma$ finite. Moreover, $\pd_{RWH}B(RWH)\leq\pd_{R\Gamma} B(R\Gamma)$ (\cite{krophollermislin} Lemma 7.3). Thus it suffices to
apply Theorem \ref{dimensionnilpotent}.
\end{proof}

There are in the literature many examples of virtually torsion free groups $\Gamma$ having $\vcd_R\Gamma=\underline{\cd}_R\Gamma$. Using any of those it is very easy to construct an example having all its finite subgroups cyclic $p$-groups thus showing that the lower bound in Lemma \ref{precyclic} can be sharp.
In Example \ref{sharpbound} we will construct a group that shows that also the upper bound can be sharp, thus completing the proof of Theorem A. To do that we will be use the next results.

\begin{lemma}\label{cyclic} Let $\Gamma$ be a group such that all finite subgroups of $\Gamma$ are cyclic $p$-groups of bounded order. Let $M=\text{Ker}(w)$ where $w$ is the augmentation 
$$w:\bigoplus_{Q\in\AA_1/\Gamma}R\uparrow_{N_\Gamma(Q)}^\Gamma\to R.$$
Then $$\pd_{R\Gamma} M+1=\pd_{R\Gamma}\Sigma\widetilde{\AA}_{1\bullet}=\pd_{R\Gamma}\Sigma\widetilde\FF_{1\bullet}\leq\pd_{R\Gamma} B(R\Gamma)+1.$$
\end{lemma} 
\begin{proof}  Lemma \ref{quillen} implies that 
$|\FF_1|\simeq_\Gamma|\AA_1|$
so $\pd_{R\Gamma}\Sigma\widetilde\FF_{1\bullet}=\pd_{R\Gamma}\Sigma\widetilde{\AA}_{1\bullet}$. The chain complex $\Sigma\widetilde{\AA}_{1\bullet}$ is
$$\ldots \to 0\to\bigoplus_{Q\in\AA_1/\Gamma}R\uparrow_{N_\Gamma(Q)}^\Gamma\buildrel w\over\to R.$$
Obviously, this chain complex is quasi-isomorphic (via inclusion) to the chain complex having only $M$ concentrated at degree 1. 
The inequality follows from Lemma \ref{lowerbound}
\end{proof}

\begin{theorem}\label{cyclic2} Let $\Gamma=K\ltimes G$ with $K$ a finite cyclic $p$-group and $G$ torsion free with $n=\cd_RG<\infty$. Then
$\underline{\cd}_R\Gamma=n+1$ if and only if there is some $T\leq\Gamma$ finite and some $RC_G(T)$-module $U$ such that
$$\text{res}:\rm{H}^n(C_G(T),U)\to \prod_{Q\in\AA_1(WT)/C_G(T)}\rm{H}^n(C_G(Q),U)$$
is not an epimorphism (in particular, $n=\cd_R C_G(Q)=\cd_R C_G(T)$ for some $T<Q$ finite).
\end{theorem}
\begin{proof} By Theorem \ref{dim2}, the condition $\underline{\cd}_R\Gamma=n+1$ is equivalent to the fact that for some finite $T$,
$$\pd_{RWT}\Sigma\widetilde\AA_1(WT)_\bullet=n+1.$$
As $\pd_{RWT}\Sigma\widetilde\AA_1(WT)_\bullet\leq\cd_RC_G(T)+1$ by Lemma \ref{cyclic} this already implies that $n=\cd_RC_G(T)$ (recall that  in this case $\pd_{RWT}B(RWT)=\vcd_RWT=\cd_RC_G(T)$). Moreover for any $T\leq Q$, $Q\leq N_\Gamma(T)$ and $C_G(Q)\leq C_G(T)$ so to simplify notation we may factor out $T$ and assume $T=1$. Note also that as the index $|\Gamma:G|$ is finite, a Serre-type argument shows that the projective dimensions of any $\Gamma$-module or $\Gamma$-chain complex of finite $\Gamma$-projective dimension with respect to both $\Gamma$ and $G$ are equal.

By Lemma \ref{cyclic}, $\pd_{RG}\Sigma\widetilde\AA_{1\bullet}=n+1$ if and only if $\pd_{RG}M=n$ where $M$ fits in the short exact sequence
(here we are seeing this as a sequence of $G$-modules, by the remark above this does not affect the projective dimensions)
$$0\to M\rightarrow\bigoplus_{Q\in\AA_1/G}R\uparrow_{C_G(Q)}^G\rightarrow R\to 0.$$
Consider the associated
 long exact sequence of Ext functors for any $RG$-module $U$. Taking into account that 
$$\text{Ext}_{RG}^i(\bigoplus_{{\AA}_1/G}R_{C_G(Q)}\uparrow^G,U)=\prod_{\AA_1/G}\text{H}^i(C_G(Q),U)$$
we see that the last nontrivial terms of that sequence are
$$\text{H}^n(G,U)\to\prod_{{\AA}_1/G}\text{H}^n(C_G(Q),U)\to\text{Ext}_{RG}^n(M,U)\to 0$$
where the first map is restriction. This yields the claim.
\end{proof}

In the next result we consider the same kind of groups but with the extra assumption that $\Gamma$ satisfies the Bredon analogue to the ordinary $\FP_\infty$ condition. This is called
Bredon $\FP_\infty$ or, for the case of the family $\FF$ of finite subgroup which is the one we are considering, $\UFP_\infty$. The easiest way to define this condition is as follows:

\begin{definition}\label{bredonfpinfty}{\rm 
A group $\Gamma$ is of type $\UFP_\infty$ over $R$ if it has finitely many conjugacy classes of finite subgroups and for any $H\in\FF$, the Weyl group $WH$ is of type $\FP_\infty$ over $R$.}\end{definition}
With the right notion of finitely generated contramodule, this definition can be shown to be equivalent to the fact that the trivial  contramodule has a resolution by finitely generated projective contramodules. In particular, if there is a finite type model for $\egamma$ then $\Gamma$ is of type $\UFP_\infty$  (for these and other related Bredon cohomological finiteness conditions, see \cite{kmn} and \cite{lueck}).

Observe also that it is a consequence of  \cite{Brown} IX Lemma 13.2 that if $\Gamma$ is virtually torsion free of type $\FP_\infty$ 
and all its finite subgroups are $p$-groups, then $\Gamma$ has finitely many conjugacy classes of finite subgroups.

\begin{corollary}\label{cyclic3} Let $\Gamma=K\ltimes G$ with $K$ a finite cyclic $p$-group and $G$ torsion free with $n=\cd_RG<\infty$. Assume that $\Gamma$ is of type $\UFP_\infty$ over $R$. Then
$\underline{\cd}_R\Gamma=n+1$ if and only if for some $T\leq \Gamma$ finite
$$\text{res: }\rm{H}^n(C_G(T),RC_G(T))\to \bigoplus_{Q\in\AA_1(WT)}\rm{H}^n(C_G(Q),RC_G(Q))$$
is not an epimorphism.
\end{corollary}
\begin{proof} We continue the proof of Theorem \ref{cyclic2}  (in particular, we only have to deal with the case $T=1$ in which $C_G(T)=G$) with the extra assumption that $\Gamma$ is of type $\UFP_\infty$ over $R$. 
Let $Q\in\AA_1$ and observe that as $C_G(Q)$ is of type $\FP$, by  \cite{Brown} VIII Proposition 6.8 there is a $G$-isomorphism
$$\begin{aligned}\text{H}^n(C_G(Q),RG)
=\text{H}^n(C_G(Q),RC_G(Q))\otimes_{C_G(Q)}RG\\
=\text{H}^n(C_G(Q),RC_G(Q))\uparrow_{C_G(Q)}^G\\
\end{aligned}$$

This means that we only have to prove that if the map 
$$\text{res: }\text{H}^n(G,RG)\to\bigoplus_{Q\in\AA_1/G}\text{H}^n(C_G(Q),RG) =\bigoplus_{Q\in\AA_1}\text{H}^n(C_G(Q),RC_G(Q))$$
 is an epimorphism, then $\underline{\cd}_R\Gamma=n$.

Now, take an arbitrary $RG$-module $U$. Applying the right exact functor  $(-)\otimes_{RG}U$ and using again  \cite{Brown} VIII Proposition 6.8  we get an epimorphism 
$$\begin{aligned}
\text{H}^n(G,U)= \text{H}^n(G,RG)\otimes_{RG}U\buildrel{\text{res}}\over\rightarrow
\bigoplus_{Q\in\AA_1/G}\text{H}^n(C_G(Q),RG)\otimes_{RG}U\\
=\bigoplus_{Q\in\AA_1/G}\text{H}^n(C_G(Q),RC_G(Q))\otimes_{RC_G(Q)}U\\
=\bigoplus_{Q\in\AA_1/G}\text{H}^n(C_G(Q),U)\\
\end{aligned}$$
It suffices to use Theorem \ref{cyclic2}.
\end{proof}

\medskip

In a celebrated paper (\cite{learynucinkis}), Leary and Nucinkis construct many interesting examples of groups using the famous Bestvina-Brady construction. One of their examples consists of a group $\Gamma$ having $\vcd\Gamma=2$ and $\underline{\cd}\Gamma=3$ thus proving that these two invariants may differ. Their group $\Gamma$ is a semi direct product of a torsion free group by the alternating group $A_5$ thus it can not be used to show that the upper bound in the second equation of Theorem A can be sharp. Here we perform a similar construction but to get groups $\Gamma$ as before such that their non-trivial finite subgroups have all order $p$ for a fixed prime $p$. The construction is based in a result of Jones, see \cite{Jones}, who proved that any $F$-acyclic complex where $F$ is a field of prime characteristic $q$ is the fixed point of the action of a cyclic group of order coprime with $q$ on a contractible space.

\begin{example}\label{sharpbound}
{\rm Consider the $CW$-structure of $\R P^2$ with only one cell for each dimension 
and denote by $l$ and $c$ the 2- and 1-dimensional cells respectively so that $\delta(l)=2c$.
Take  an odd prime $p$ (any odd integer would do) and let $Q$ be the cyclic group of $p$ elements. By \cite{Jones} Theorem 1, $\R P^2$ can be embedded in a contractible finite $CW$-complex with a $Q$-action so that the fixed points subspace is precisely $\R P^2$. Explicitly the complex can be constructed as follows: Let $Q$ act trivially on $\R P^2$ and
add $p$ free $Q$-cells $s_0,\ldots,s_{p-1}$ on dimension $2$ that kill the 1-homology, i.e. with $\delta(s_i)=c$.
Then the new homology group in degree 2 is a free $Q$-module. In fact, it is the free $\Z$-module on $l-s_i-s_{i+1}$ for $i=1,\ldots,p$ where the subindices are taken mod $p$. So we may also kill this new homology by adding free $Q$-cells on dimension 3.
Let $L$ be a flag triangulation of the contractible complex constructed above. Then $L$ is a contractible flag complex with a simplicial admissible $Q$-action (admissible means that if a simplex is fixed, all its faces are also fixed) and the fixed points  subcomplex $L^Q$ has the same cohomology  as  $\R P^2$.

Let $G=H_L$ be the Bestvina-Brady group associated to $L$. The action of  $Q$ on $L$ induces an action of $Q$ on $G$ so we may define $\Gamma=Q\ltimes G$. Then by \cite{learynucinkis} Theorem 3,  $C_G(Q)=H_{L^Q}$, in particular  $C_G(Q)$ is also a Bestvina-Brady group. Moreover the fact that $L^Q\neq\emptyset$ implies that all the nontrivial finite subgroups of $\Gamma$ are conjugated. 
 As $\text{dim}L=3$ and it is contractible, Theorem 22 of \cite{IanMuge} implies that $\text{cd}G=3$. On the other hand, 
 $\text{H}^2(L^Q,\Z)\neq 0$ thus the same result implies $\text{cd}H_{L^Q}=3$.

Now denote by $Z$ the infinite cyclic group. By \cite{IanMuge} Corollary 12 there is a short exact sequence
$$0\to\prod_Z\text{H}^2(L^Q,\Z)\to\text{H}^3(H_{L^Q},\Z)\to U\to 0$$
for certain $\Z[Z]$-module $U$. As $\text{H}^2(L^Q,R)=\Z_2$ we see that $\text{H}^3(H_{L^Q},\Z)$ can not be finitely generated as an abelian group.  On the other hand, by the original result of Bestvina-Brady (see \cite{learynucinkis} Theorem 1) the fact that $L$ is contractible implies that $G$ is of type $\FP_\infty$ thus $\text{H}^3(G,\Z)$ is finitely generated (as an abelian group). This implies that 
$$\text{res}:
\text{H}^3(G,\Z)\to\text{H}^3(C_G(Q),\Z)=\text{H}^3(H_{L^Q},\Z)$$
 is not an epimorphism so by Theorem \ref{cyclic2}, $\underline{cd}\Gamma=\cd G+1=4$.

However $L^Q$ is not acyclic thus the group $H_{L^Q}=C_{H^L}(Q)$ is not of type $\FP_\infty$ thus $\Gamma$ is not $\UFP_\infty$. Also, $\Gamma$ is of type $\UFP_\infty$ over any ring $R$ where $2$ can be inverted, moreover, if $2$ can be inverted in $R$, then $\cd_RG=\underline{\cd}_R\Gamma=3$.

This can be generalized as follows. Let $q$ be a prime and let $M$ be the Moore space obtained from the $m$-sphere attaching an $m+1$-cell via a degree $q$ attaching map, so $M$ has dimension $m+1$. Let $Q=C_p$ be a finite cyclic group of prime order $p\neq q$. Use Jones procedure to embed $M$ in a finite $CW$-complex $T$ of dimension $m+2$ such that $T$ is contractible and $Q$ acts cellularly on $T$  with $T^Q=M$. Now, let $L$ be a flag triangulation of $T$ and consider the associated Bestvina-Brady group $G=H_L$.  As before we may form the semi direct product $\Gamma=Q\ltimes G$ and the same argument shows
$$\underline{cd}\Gamma=\cd G+1.$$}
\end{example}

\begin{remark}{\rm By \cite{learynucinkis},  the Bestvina-Brady construction can not be used to produce examples of groups $\Gamma$ with $\vcd\Gamma<\underline\cd\Gamma$ and $\Gamma$ being of type $\UFP_\infty$. }
\end{remark}

So far in this Section we have considered mainly groups whose finite subgroups have the simplest possible structure. In particular, the
chain complex associated to the poset of non-trivial subgroups was fairly easy to describe (this has also been exploited in \cite{langerluck} and \cite{yagoramon}). Things change dramatically if one wants to include more complicated finite subgroups. In the next result we analyze one of the properties of the poset of finite subgroups in the next easiest case, i.e., a semi direct product of a torsion free group by a $p$-elementary abelian group of rank 2.

\begin{theorem}\label{acyclic} Let $\Gamma=K\ltimes G$ with $G$ torsion free and $K\cong C_p\times C_p$. Assume that for any $G\leq S\leq\Gamma$ there is only one conjugacy class of finite subgroups $H$ of $\Gamma$ such that $HG=S$.
Then $|\FF_1|$ is 2-spherical or acyclic if and only if $G=\langle C_G(Q)\mid 1<Q<K\rangle$. 
\end{theorem}
\begin{proof} Let $B_\bullet$ be the $\Z\Gamma$-chain complex associated to the $\Gamma$-set of subgroups of order $p^2$. There is a short exact sequence of chain complexes 
$$0\to\Sigma\widetilde B_\bullet\to\Sigma\widetilde\FF_\bullet\to C_\bullet \to 0$$
for certain $C_\bullet$ of length 2 vanishing on degree 0. Then $C_2$ is the free $R$-module on the set of pairs $H_1<H_2\leq\Gamma$ with
$|H_1|=p,$ $|H_2|=p^2$ whereas $C_1$  is the free $R$-module on the set of subgroups of $\Gamma$ or order $p$ and the differential $\bar\delta:C_2\to C_1$ maps $1_{H_1<H_2}$ onto $1_{H_1}$. The hypothesis on $\Gamma$ implies that any $H\leq\Gamma$ of order $p$ is contained in some subgroup of order $p^2$ thus $\text{H}_1(C_\bullet)=0$.
From the long exact sequence of homology groups we get:
$$0\to\text{H}_2(\Sigma\widetilde\FF_\bullet)\to\text{H}_2(C_\bullet)\to\text{H}_1(\Sigma\widetilde B_\bullet)\buildrel f\over\to
\text{H}_1(\Sigma\widetilde\FF_\bullet)\to  0.$$
Then, $|\FF_1|$ is 2-spherical or acyclic if and only if $f=0$. The hypothesis implies that $\text{H}_1(\Sigma\widetilde B_\bullet)$ is freely generated as an abelian group by the elements of the form $1_K-1_{K^g}$ where $g$ runs over a set of representatives of the classes in $G/C_G(K)$, therefore the map $f$ is zero if and only if for any $g\in G$, $1_K-1_{K^g}\in\text{Im}\delta$ where $\delta$ is the map
$$\bigoplus_{\{Q<H\mid Q,H\in\AA_1\}}\Z\buildrel\delta\over\rightarrow\bigoplus_{\{Q\mid Q\in\AA_1\}}\Z$$
with $\delta(1_{Q<H})=1_H-1_Q$.

Assume first that $1_K-1_{K^g}\in\text{Im}\delta$ for any $g\in G$. Then there is a  chain of elements $g_1,\ldots,g_t\in G$ and of subgroups $Q_1,\ldots,Q_t\leq K$ of order $p$
such that $g=g_t$ and 
\begin{equation}\label{delta}\delta(1_{Q_1<K}-1_{Q_1<K^{g_1}}+1_{Q_2^{g_1}<K^{g_1}}-1_{Q_2^{g_1}<K^{g_2}}+\ldots-1_{Q_t<K^{g_t}})=1_K-1_{K^g}.\end{equation}
This implies 
$Q_i^{g_{i-1}}\leq K^{g_{i-1}}\cap K^{g_i}$
and putting $g_0=1$, we have $Q_i\leq K\cap K^{g_ig_{i-1}^{-1}}$ for $i=1,\ldots,t$. Therefore by Lemma 3.4 of \cite{conch2}, $g_ig_{i-1}^{-1}\in C_G(Q_i)$ thus 
$$g=g_t=g_tg_{t-1}^{-1}g_{t-1}g_{t-2}^{-1}\ldots g_2g_1^{-1}g_1g_0^{-1}\in \langle C_G(Q)\mid 1<Q<K\rangle.$$

Conversely, assume $G=\langle C_G(Q)\mid 1<Q<K\rangle$. Let $g\in G$ and put $g^{-1}=h_1\ldots h_t$ where each $h_i\in C_G(Q_i)$ for some $1<Q_i<K$. Then $Q_i\leq K\cap K^{{h_i}^{-1}}$ and we can form a chain as in (\ref{delta}) with $g_0:=1$ and $g_i:=h_i^{-1}g_{i-1}$ for $i=1,\ldots,t$ mapping onto $1_K-1_{K^g}$.

\end{proof}

\begin{remark}{\rm Basically with the same proof it can be deduced that the condition in the hypothesis is equivalent to $\text{H}_1(\Sigma\widetilde\FF_\bullet)=0$ even if $|K|=p^n$ for some $n$.}
\end{remark}

\begin{example}\label{artinconj} {\rm Let $\Gamma=K\ltimes A_L$ where $K$ is a finite group acting admissibly and simplicially on the flag complex $L$ and $A_L$ is the right-angled Artin group associated to $L$. Then for each subgroup $A_L\leq S\leq\Gamma$  there is a single conjugacy class of finite subgroups $H\leq \Gamma$ such that $HA_L=S$. This is a consequence of the proof of Theorem 3 in \cite{learynucinkis}. The explicit argument is as follows: Observe that we may choose $H\leq K$ with $HA_L=S$ and let $X$ denote the model for the Eilenberg-Mac Lane space $K(A_L,1)$ whose $n$-cells are $n$-cubes corresponding to pairs $(gx_0,\sigma)$ where  $x_0\in X$ is a fixed vertex; $g\in A_L$ and $\sigma$ is an $(n-1)$-simplex of $L$ (see \cite{learynucinkis} Section 3). Then $X$ is a $\text{CAT}(0)$-space with a $\Gamma$ action such that for $t\in K$; $(gx_0,\sigma)^t=(g^tx_0,\sigma^t)$.  The fact that $X$ is $\text{CAT}(0)$ implies that for any finite $H$ as before,  $X^H$ is contractible, in particular non-empty thus for some $g\in A_L$, $gx_0\in X^H$ which means that $g^{-1}Hg\leq K$ (the contractibility of the $X^H$'s also implies the well known fact that $X$ is a model for $\egamma$). 
Moreover  by \cite{learynucinkis} Theorem 2 (due to Crisp), $C_{A_L}(K)=A_{L^K}$.

Now, let $L$ be the simplicial complex represented by

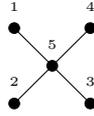
\centerline{   \begin{tikzpicture}
 \draw[black]
    (5,0) -- (5.5, 0.5) -- (6,0);
    \draw[black] (5,1) -- (5.5,0.5) --(6,1);
 \filldraw (5,0) circle (2pt) node[above=3pt]{\tiny 2};
  \filldraw (6,0) circle (2pt) node[above=3pt]{\tiny 3};
 \filldraw (5,1) circle (2pt) node[above=3pt]{\tiny 1};
 \filldraw (6,1) circle (2pt) node[above=3pt]{\tiny 4};
  \filldraw (5.5,0.5) circle (2pt) node[above=3pt]{\tiny 5};
   \end{tikzpicture}}
   
   \bigskip
\noindent and let the following permutation groups act on $L$ by permuting the labelled vertices:
$$K_1:=\langle(1,2),(3,4)\rangle;$$
$$K_2:=\langle(1,2)(3,4),(1,3)(2,4)\rangle.$$

Then Theorem \ref{acyclic} implies that $\text{H}_1(|\FF_1(K_1\ltimes A_L)|)=0\neq\text{H}_1(|\FF_1(K_2\ltimes A_L)|).$}
\end{example}

\section{Euler classes}\label{euler}

\noindent In this Section we shall concentrate on equivariant Euler classes of virtually torsion free groups $\Gamma$ of type $\UF$ having all their finite subgroups either nilpotent or $p$-groups or cyclic $p$-groups. Recall that $\Gamma$ is of type $\UF$ if there is a cocompact $\egamma$, this is just the Bredon analogue of the property  $\F$ of admitting a cocompact model of $\operatorname{E}\Gamma$. 
Many interesting families of groups such as for example word hyperbolic groups, solvable groups of type $\FP_\infty$, mapping class groups and finite extensions of right-angled Artin groups by a finite group of admissible simplicial automorphisms of the associated flag complex are known to be of type $\UF$. Any group of type $\UF$ is also of type $\UFP_\infty$ so it has only finitely many conjugacy classes of finite subgroups.

Using Hattori-Stallings ranks, we could get analogous results for groups of type $\UFP$, i.e., groups which are $\UFP_\infty$ and have $\underline{\cd}\Gamma<\infty$ (equivalently, groups for which there is a finite length resolution of the trivial contramodule by finitely generated projective Bredon contramodules). However, for simplicity we not consider here that case and moreover we work only with coefficients in $\Z$ in this Section.

We use the same notation as in \cite{conch2}, see also \cite{LuckL2} Definition 6.84. So we let $\Gamma$ be a group of type $\UF$ and put
$$\chi^\Gamma(\egamma):=\sum_{H\in\FF/\Gamma}e^{WH}(\Sigma\widetilde\FF_{H\bullet})$$
where for a group $S$, $e^S(C_\bullet)$ denotes the equivariant Euler characteristic of an $S$-chain complex $C_\bullet$ having an $S$-finite $S$-free resolution (i.e., the alternating sum of the $S$-ranks of the free modules of that resolution). The same applies for $S$-modules, seen as chain complexes concentrated in degree 0. We will use the following properties of these Euler characteristics: They are invariant under $S$-quasi-isomorphisms and if $C_\bullet$ is a chain complex of length $n$, 
\begin{equation}\label{decomposition}e^S(C_\bullet)=\sum_{i=0}^ne^S(C_i)
\end{equation}
 when all these Euler characteristics are defined. Moreover,
if $S\leq T$, 
\begin{equation}\label{induction}e^T(C_\bullet\uparrow_S^T)=e^S(C_\bullet)\end{equation}
and in the case when the index $|T:S|$ is finite we also have for a $T$-complex $D_\bullet$,
\begin{equation}\label{finindex}e^T(D_\bullet)={1\over|T:S|}e^S(D_\bullet),\end{equation}
again assuming that $e^T(D_\bullet)$ is defined.
(For these and many other properties of equivariant Euler characteristics, see \cite{Brown} IX). Using (\ref{finindex}), we may extend slightly the definition of Euler characteristic of a chain complex:

\begin{definition}{\rm Let $\Gamma$ be a group with $G\leq\Gamma$ torsion free of finite index and let $C_\bullet$ be a $\Gamma$-chain complex (or module) admitting a finite free resolution as $G$-chain complex. Then
$$e^\Gamma(C_\bullet):={1\over|\Gamma:G|}e^G(C_\bullet).$$}
\end{definition}
One easily checks that all the properties above remain true for this more general version. In fact
this is exactly what is done in \cite{Brown} IX to define Euler characteristics of groups $\Gamma$ which are virtually of type $\F$. With the same notation as there we  put for $\Gamma$ virtually $\F$:
$$\chi(\Gamma):=e^\Gamma(\Z).$$
We also use $\chi$ for the ordinary Euler characteristic of a finite simplicial complex and $e$ for the ordinary Euler characteristic of a finite complex of finitely generated free abelian groups.

For an arbitrary group $\Gamma$ we set
$$\Omega:=\{H\in\FF/\Gamma\mid N_\Gamma(H)\text{ is not commensurable with }N_\Gamma(K)\text{ for any }H<K\}.$$ 

\begin{lemma}\label{eulernilpotent} Let $\Gamma$ be a group of type $\UF$ with all its finite subgroups nilpotent. Then for any $H\leq\Gamma$ finite, the coefficient of $[\Gamma/H]$ in $\chi^\Gamma(\egamma)$ equals the coefficient of $[WH/1]$ in $\chi^{WH}(\underline{\text{E}}WH)$. Thus
$$\chi^\Gamma(\underline{E}\Gamma)
=\sum_{H\in\Omega}e^{WH}(\Sigma\widetilde{\AA}_1(WH)_\bullet)[\Gamma/H].$$
\end{lemma}
\begin{proof} By Lemmas \ref{nilpotent} and \ref{quillen} the coefficient of $[\Gamma/H]$ in $\chi^\Gamma(\egamma)$ is
$$e^{WH}(\Sigma\widetilde\FF_{H\bullet})=e^{WH}(\Sigma\widetilde\FF_{1}(WH)_\bullet)=e^{WH}(\Sigma\widetilde\AA_{1}(WH)_\bullet).$$
The fact that the sum can be restricted to $\Omega$ follows from Lemma \ref{contractible}.
\end{proof}

We state the form that the formula of Lemma \ref{eulernilpotent} has in the case when all the finite subgroups are cyclic $p$-groups.

\begin{theorem}\label{eulercyclic} Let $\Gamma$ be a virtually torsion free group of type $\UF$ with all its finite subgroups  cyclic $p$-groups. Then
$$\chi^\Gamma(\egamma)=\sum_{H\in\Omega}\Big(\chi(WH)-\sum_{Q\in\AA_1(WH)/WH}\chi(N_\Gamma(Q)/H)\Big)[\Gamma/H]$$
\end{theorem}
\begin{proof} Note that the fact that all the finite subgroups of $\Gamma$ are cyclic implies that whenever $H<Q\in\FF$, then $N_\Gamma(Q)\leq N_\Gamma(H)$. It suffices to take into account that the chain complex $\Sigma\widetilde\AA_1(WH)_\bullet$ is 
$$0\to\bigoplus_{Q\in\AA_1(WH)/WH}\Z\uparrow_{N_\Gamma(Q)/H}^{WH
}\to\Z\to 0$$
and use (\ref{decomposition}) and (\ref{induction}).\end{proof}

\begin{remark}{\rm If $H\leq\Gamma$ is a subgroup such that $WH$ has all its finite subgroups cyclic $p$-groups, then the coefficient of $[\Gamma/H]$ in $\chi^\Gamma(\egamma)$ can be computed with an analogous formula.}
\end{remark}

To illustrate the formula of Theorem \ref{eulercyclic} we are going to consider the example of a cyclic finite extension of a right-angled Artin group. Note that if one is interested only in this case, another way to proceed in order to compute $\chi^\Gamma(\egamma)$ is to use the model for $\egamma$ described in Example \ref{artinconj}.

\begin{example} {\rm Let $\Gamma=K\ltimes A_L$ where $K=\langle x\mid x^{p^n}=1\rangle$ is the cyclic group of order $p^n$ acting admissibly and simplicially on a flag complex $L$ and $A_L$ is the associated right-angled Artin group .  The argument quoted near the beginning of Example \ref{artinconj} implies that
 for each $G\leq S\leq \Gamma$ there is only one conjugacy class of finite subgroups $H\leq\Gamma$ with $HG=S$. In particular for any $H\leq\Gamma$ finite, $H$ is conjugate to a subgroup of $K$ thus we may assume $H\leq K$ and then $N_\Gamma(H)=K\ltimes C_G(H)=K\ltimes A_{L^H}$, therefore $WH=K/H\ltimes A_{L^H}$ is a group of the same kind. Denote $K_i=\langle x^{p^i}\rangle$. The formula of Theorem \ref{eulercyclic} and (\ref{finindex}) yield
$$\chi^\Gamma(\egamma)=\chi(A_{L^{K}})[\Gamma/K]+
\sum_{i=1}^n{1\over p^i}\Big(\chi(A_{L^{K_i}})-\chi(A_{L^{K_{i-1}}})\Big)[\Gamma/K_i].$$
Euler characteristics of right-angled Artin groups were computed in \cite{charneydavis} Corollary 2.2.5 where it is proven that $\chi(A_L)=1-\chi(L)$. Therefore
$$\chi^\Gamma(\egamma)=\Big(1-\chi(L^{K})\Big)[\Gamma/K]+
\sum_{i=1}^n{1\over p^i}\Big(\chi(L^{K_{i-1}})-\chi(L^{K_i})\Big)[\Gamma/K_i].$$}

\end{example}

We proceed now to the proof of Theorem B. First, we observe that it follows from \cite{lueck} Theorem 4.2 that a virtually torsion free group $\Gamma$ is of type $\UF$ if and only if  all the Weyl groups $WH$ for $H\in\FF$ are virtually of (ordinary) type $\F$.

\begin{theorem}\label{formulaeuler} Assume that $\Gamma=K\ltimes G$ is of type $\UF$ with $G$ torsion free and $K$ a finite $p$-group. 
Assume moreover that for any $G\leq S\leq\Gamma$ all those finite subgroups $H\leq\Gamma$ with $HG=S$ are $\Gamma$-conjugate. Then the coefficient of $[\Gamma/1]$ in $\chi^\Gamma(\egamma)$ is
$$\sum_{H\in\AA(K)/K}{(-1)^{\text{lg}_p|H|}p^{\Big({\text{lg}_p|H|\atop 2}\Big)}\over|N_K(H)|}\chi(C_G(H))$$
where we adopt the convention that $\Big({0\atop 2}\Big):=\Big({1\atop 2}\Big):=0$
\end{theorem}
\begin{proof} By Lemma \ref{nilpotent}, the coefficient of $[\Gamma/1]$ in $\chi^\Gamma(\egamma)$ is 
$e^\Gamma(\Sigma\widetilde{\mathcal{A}}_{1\bullet}).$ 
 We may decompose each degree of  the complex of $\Sigma\widetilde{\mathcal{A}}_{1\bullet}$ as a sum of modules 
 according to the $\Gamma$-conjugacy class of the top subgroup of each chain. 
 The hypothesis implies that we may choose as representatives of the $\Gamma$-conjugacy classes of finite elementary abelian subgroups of $\Gamma$ just a family of representatives of the $K$-conjugacy classes of elementary abelian subgroups of $K$. 
Let us denote by $\mathcal{S}(H)$  the poset of proper subgroups of the $p$-elementary abelian group $H$ and by $\widetilde{\mathcal{S}}(H)_\bullet$ the corresponding augmented chain complex. Note that a chain $\sigma:K_0<\ldots K_{i-1}<K_i=H$ in $\AA_1$ shows on in $\Sigma\widetilde{\mathcal{A}}_{1\bullet}$ precisely at degree $i+1$ and corresponds to a degree $i-1$ chain in $\widetilde{\mathcal{S}}(H)_\bullet$ (we are not suspending this complex, just augmenting it). Note also that each such chain yields exactly $|\Gamma:N_\Gamma(H)|$ chains of the original complex with top subgroup lying in the same conjugacy class. Also, for any $i$, $\widetilde{\mathcal{S}}(H)_i\uparrow_{N_\Gamma(H)}^\Gamma\downarrow_G$ is a sum of $G$-modules which are $\Gamma$-conjugated to
$\widetilde{\mathcal{S}}(H)_i\uparrow_{C_G(H)}^G$. As $C_G(H)$ is of type $F$ and acts trivially on $\widetilde{\mathcal{S}}(H)_i$, we see that this module has a finite $C_G(H)$-free resolution. 
Therefore

$$\begin{aligned}
e^\Gamma(\Sigma\widetilde{\mathcal{A}}_{1\bullet})&\underset{(\ref{decomposition})}{=}e^\Gamma(\Z)+\sum_{H\in\AA_1(K)/K}e^\Gamma(\widetilde{\mathcal{S}}(H)_\bullet\uparrow_{N_\Gamma(H)}^\Gamma)\\
&\underset{(\ref{induction}),(\ref{finindex})}{=}{1\over|K|}\chi(G)+\sum_{H\in\AA_1(K)/K}{1\over|N_K(H)|}e^{C_G(H)}(\widetilde{\mathcal{S}}(H)_\bullet)\\
&={1\over|K|}\chi(G)+\sum_{H\in\AA_1(K)/K}{e(\widetilde{\mathcal{S}}(H)_\bullet)\over|N_K(H)|}\chi(C_G(H))\\
\end{aligned}$$
where 
in the first equality we are applying (\ref{decomposition}) twice: one to decompose $\Sigma\widetilde{\mathcal{A}}_{1\bullet}$ in terms of the modules $\widetilde{\mathcal{S}}(H)_i\uparrow_{N_\Gamma(H)}^\Gamma$ and a second one to recover the complex $\widetilde{\mathcal{S}}(H)_\bullet\uparrow_{N_\Gamma(H)}^\Gamma$.

If $|H|=p$, $\mathcal S(H)=\emptyset$ thus  $e(\widetilde {\mathcal S}(H))=-1$. And if
 $|H|=p^n$ for $n\geq 2$ this Euler characteristic is well known to be (see \cite{stanley} Example 3.10.2)
$$e(\widetilde{\mathcal{S}}(H)_\bullet)=(-1)^np^{\big({n\atop 2}\big)}.$$
From this we get the result.
\end{proof}



\begin{corollary}\label{eulerraag} Let $K\neq 1$ be a finite $p$-group acting admissibly and simplicially on the flag complex $L$ and $\Gamma=K\ltimes A_L$. Then the coefficient of $[\Gamma/1]$ in $\chi^\Gamma(\egamma)$ is
$$-\sum_{H\in\AA(K)/K}{(-1)^{\text{lg}_p|H|}p^{\Big({\text{lg}_p|H|\atop 2}\Big)}\over|N_K(H)|}\chi(L^H)$$
\end{corollary}
\begin{proof} By the beginning of Example \ref{artinconj}, $\Gamma$ (and all its Weyl groups) satisfies the hypothesis of Theorem \ref{formulaeuler}. Moreover, by \cite{charneydavis} Corollary 2.2.5,  $\chi(A_{L^H})=1-\chi(L^H)$ for any $H\leq\Gamma$ finite. Therefore the Corollary follows from Theorem \ref{formulaeuler} once we have proven 
$$0=\sum_{H\in\AA(K)/K}{(-1)^{\text{lg}_p|H|}p^{\Big({\text{lg}_p|H|\atop 2}\Big)}\over|N_K(H)|}=\sum_{H\in\AA_1(K)/K}{e(\widetilde{\mathcal{S}}(H)_\bullet)\over|N_K(H)|}+{1\over|K|}.
$$
But observe that
$$\sum_{H\in\AA_1(K)/K}|K:N_K(H)|{e(\widetilde{\mathcal{S}}(H)_\bullet)}=e(\AA_1(K)_\bullet)=1$$
since $|\AA_1(K)|$ is contractible by a well known result of Quillen (in fact it is conically contractible, see \cite{Ben2})
\end{proof}

\begin{example} {\rm Let $K=C_p\times C_p$ act admissibly and simplicially on a flag complex $L$ and put $\Gamma=K\ltimes A_L$. Then the coefficient of $[\Gamma/1]$ in 
$\chi^\Gamma(\egamma)$ is
$$-{1\over p^2}\chi(L)+{1\over p^2}\sum_{1<H<K}\chi({L^H})-{1\over p}\chi({L^K}).$$}
\end{example}

If the hypothesis in Theorem \ref{formulaeuler} are inherited by centralizers, then the formula together with Lemma \ref{eulernilpotent} allows to compute the whole Euler class $\chi^\Gamma(\egamma)$. We do it in the next example:

\begin{example} {\rm Let $K=\langle x,y\mid x^2=y^4=1,y^x=y^{-1} \rangle$ be a dihedral group acting admissibly and simplicially on a flag complex $L$ and put $\Gamma=K\ltimes A_L$. Then the coefficient of $[\Gamma/1]$ in 
$\chi^\Gamma(\egamma)$ is
$$-{1\over 8}\chi(L)+{1\over 4}\sum_{H\in\Omega_1}\chi({L^H})
+{1\over 8}\chi({L^{\langle y^2\rangle}})-{1\over 4}\sum_{H\in\Omega_2}\chi({L^H})$$
where $\Omega_1=\{\langle x\rangle,\langle xy\rangle\}$ and $\Omega_2=\{\langle x, y^2\rangle,\langle xy,y^2\rangle\}$.

For $H\in\Omega_1\cup\Omega_2\cup\{\langle y\rangle\}$, $|N_K(H)/H|=2$ and the coefficient of $[\Gamma/H]$ is
$$-{1\over 2}\chi({L^H})+{1\over 2}\chi({L^{N_K(H)}}).$$

For $H=\langle y^2\rangle$, $N_K(H)=K$, $N_K(H)/H\cong C_2\times C_2$ and the coefficient of $[\Gamma/H]$ is
$$-{1\over 4}\chi({L^H})+{1\over 4}\sum_{S\in\Omega_2\cup\{\langle y\rangle\}}\chi({L^{S}})-{1\over 2}\chi({L^K}).$$
 
 Finally, the coefficient of $[\Gamma/K]$ is $1-\chi({L^K})$.}
\end{example}

\section{Duality and Bredon cohomological dimensions.}\label{poincareduality}

\noindent Example \ref{sharpbound} shows that a finite extension of prime index of a torsion free group of type $\FP_\infty$ might not be of type $\FP_\infty$. 
However, Hamilton (\cite{Hamm} Theorem E) has shown that this phenomenon can not happen if we work with coefficients in a field of prime characteristic  equal to the index of the extension. We are going to use Hamilton's techniques to further explore the properties of Bredon cohomology in characteristic $p$  of groups $\Gamma$ which are a $p$-power index extension of a duality or Poincar\'e duality torsion free group. The notion of Poincar\'e duality has been generalized to Bredon cohomology by Davis and Leary  and we may also generalize the notion of a duality group in the same spirit.

\begin{definition}\label{bredonPD}(\cite{davisleary})  {\rm A group $\Gamma$ is a Bredon duality group over $R$ if it is of type $\UFP_\infty$ for $R$, $\underline{\cd}_R\Gamma<\infty$ and for any finite subgroup $H$ there is an integer $n_H$ such that 
$$\text{H}^i(WH,RWH)=\Bigg\{
\begin{aligned}&\text{is $R$-flat if }i=n_H\\
&=0\text{ in other case.}\\
\end{aligned}$$
If $\Gamma$ has the same property with the extra requirement that for any $H\leq\Gamma$ finite, $\text{H}^{n_H}(WH,RWH)=R$, then we say that $\Gamma$
is a Bredon Poincar\'e duality group over $R$.}
\end{definition}

Note that for torsion free groups, this reduces to ordinary duality or Poincar\'e duality. Recall that if $G$ is a torsion free duality group, the integer $n:=n_1$ is called the dimension of $G$ and $n=\cd_RG$. We will  denote $A_R(G):=\text{H}^n(G,RG)$, this is called the dualizing module. 

We fix now a field $F$ of characteristic $p$.

\begin{proposition}\label{tec1} Let $G$ be a duality group over $F$ of dimension $n$, $\Gamma=Q\ltimes G$ with $Q=C_p$ and $U=F_Q\uparrow^\Gamma$.  Then for $m>1,$ 
$${\text{H}}^m(\Gamma/G,\text{H}^n(G,U))=\bigoplus_{H\in\FF_1/\Gamma}\bigoplus_{i=0}^n\text{H}^i(C_G(H),FC_G(H)).$$
\end{proposition}
\begin{proof}  Observe first that as $U\downarrow_G\cong RG$, the fact that $G$ is $R$-duality implies $\text{H}^j(G,U)=0$ unless $j=n$ so the Lyndon-Hochshild-Serre spectral sequence associated to the group extension $1\to G\to\Gamma\to Q\to 1$ yields
$$\text{H}^{m+n}(\Gamma, U)=\text{H}^m(\Gamma/G,\text{H}^n(G,U)).$$
Consider now the short exact sequence of Lemma \ref{cyclic}
$$0\to M\to \bigoplus_{H\in\FF_1/\Gamma}F\uparrow^\Gamma_{N_\Gamma(H)} \to F\to 0$$
and recall that $\pd_{R\Gamma} M\leq n$. From the long exact sequence associated to $\text{Ext}^*_\Gamma(-,U)$ we get
an isomorphism
$$\text{H}^{m+n}(\Gamma, U)\cong\bigoplus_{H\in\FF_1/\Gamma}\text{H}^{m+n}(N_\Gamma(H),U).$$
For any $H\in\FF_1$, the Mackey formula together with the fact that since $N_\Gamma(H)=H\times C_G(H)$, $H$ is the only finite subgroup of $N_\Gamma(H)$ imply that 
$$U\downarrow_{N_\Gamma(H)}=U_H\oplus V$$
where $U_H:=F_H\uparrow^{N_\Gamma(H)}$ and $V$ is a free $N_\Gamma(H)$-module. Then, as $m+n>n\geq\vcd N_\Gamma(H)$,
$$\text{H}^{m+n}(N_\Gamma(H),U)=\text{H}^{m+n}(N_\Gamma(H),U_H).$$
As we are working with coefficients in a field $F$ and $H=C_p$, we have $\text{H}^i(H,F)=F$. Then, using the K\"unneth Theorem exactly as in the proof of \cite{Hamm} Proposition 2.4 we get
$$\text{H}^{m+n}(N_\Gamma(H),U_H)=\bigoplus_{i=0}^n\text{H}^i(C_G(H),U_H)=\bigoplus_{i=0}^n\text{H}^i(C_G(H),FC_G(H)).$$

\end{proof}

Note that with the notation of Proposition \ref{tec1}, we may see $A_F(G)=\text{H}^n(G,U)$ as a $Q$-module and then
$\text{H}^m(Q,A_F(G))={\text{H}}^m(\Gamma/G,\text{H}^n(G,U)).$

\begin{corollary}\label{dualizingmod} If $\Gamma=Q\ltimes G$ with $Q=C_p$ is Bredon duality over $F$, then for $m>1,$  $$\text{H}^m(Q,A_F(G))=\bigoplus_{H\in\FF_1/\Gamma}A_F(C_G(H)).$$
\end{corollary}

We may now prove the analogous of Hamilton's result for Poincar\'e duality groups, which is the first part of Theorem C. As stated in the Introduction, this result generalizes Corollary 2.1 in \cite{farrelllafont}.

 \begin{theorem}\label{poincaredualityext} Let $G$ be a Poincar\'e duality group over $F$ and $\Gamma$ a finite index extension of $G$ such that $|\Gamma/G|$ is a $p$-group. Then $\Gamma$ is Bredon Poincar\'e duality over $F$ and $\underline{\cd}_F\Gamma=\cd_FG$.
 \end{theorem}
 \begin{proof} We proceed by induction exactly as in the proof of \cite{Hamm} Theorem 6.4. Then one can reduce to the case when $\Gamma=Q\ltimes G$ with $Q=C_p$.
 The fact that $\Gamma$ is Bredon Poincar\'e duality follows from Proposition \ref{tec1}. Using \cite{Bieribook} Proposition 9.22 we see that every subgroup of a Poincar\'e duality group has either finite index or it is of homological dimension strictly smaller than that of the original group. This means that we may apply Theorem \ref{condition} with $l(H):=\text{hd}_RC_G(H)$ thus $\underline\cd_F\Gamma\leq\text{hd}_FG=\text{cd}_FG$ (recall that $G$ is of type $\FP_\infty$ over $F$ thus $\text{hd}_FG=\cd_FG$ and the same holds for any $C_G(H)$).

  \end{proof}

 We consider now the Question D of the Introduction. 

\begin{examples} {\rm
\noindent i) Leary-Nucinkis (\cite{learynucinkis}) have examples where $\Gamma$ is $\FP_\infty$ and not  $\UFP_\infty$. In Section \ref{smallrank} we have constructed an example with the same property and such that $|\Gamma:G|$ is a prime. But by \cite{Hamm} this can not happen in characteristic $p$ if $\Gamma/G$ is a $p$-group. 

\noindent ii) Davis-Leary (\cite{davisleary}) have examples where $G$ is Poincar\'e duality but $\Gamma$ is not even $\UFP_\infty$ and also examples where $G$ is again Poincar\'e duality and $\Gamma$ is $\UFP_\infty$ but not Bredon Poincar\'e duality. 
And in   \cite{farrelllafont} Theorem 2.2 and the Remark after that Theorem, Farrell-Lafont  construct examples where $G$ is $\delta$-hyperbolic and Poincar\'e duality and $|\Gamma:G|$ is a prime but $\Gamma$ is not Bredon Poincar\'e duality and also examples such that $|\Gamma:G|$ is not a prime and $\Gamma$ is not Bredon Poincar\'e duality over any PID. 
Again, there is a positive result in characteristic $p$ if $\Gamma/G$ is a $p$-group (Theorem \ref{poincaredualityext}). 

\noindent iii) The examples of item ii) also show that if $G$ is a duality group, $\Gamma$ might not be Bredon duality. In Section \ref{dualityext} we construct an example where $G$ is duality over a field $F$ of characteristic $2$ and $|\Gamma:G|=2$ but $\Gamma$ is not Bredon duality over $F$, in contrast to what happens for the properties of being $\UFP_\infty$ or Bredon Poincar\'e duality.
}
\end{examples}

 \begin{example} {\rm Finite index extensions of solvable groups of type $\FP_\infty$ are of type $\UFP_\infty$, are virtually torsion free and have $\underline{\cd}\Gamma=\vcd\Gamma$(\cite{britayo}). Torsion free solvable Poincar\'e duality groups are precisely polycyclic groups (\cite{Bieribook} Theorem 9.23) and torsion free solvable duality groups are precisely those of type $\FP_\infty$ (\cite{GildenhuysStrebel}). So the class of virtually-(torsion-free solvable) groups satisfies all the conditions in Question D (and the class of elementary amenable groups satisfies conditions i), ii) and iii)). }
 \end{example}
 
 \begin{example} {\rm Virtually surface groups and virtually free groups are Bredon Poincar\'e duality and have $\underline{\cd}\Gamma=\vcd\Gamma$.
The last assertion is due to Mislin \cite{mislin}. The first one also follows basically with Mislin's proof. }
\end{example}

 \begin{question}{\rm If $\Gamma$ is virtually torsion free and of type $\UFP_\infty$, does it follows that $\underline{\cd}\Gamma=\vcd\Gamma$?
And what if we add stronger hypothesis, such as $\Gamma$ being Bredon duality, or Bredon Poincar\'e duality?}
 \end{question}

\section{Duality groups in characteristic $p$.}\label{dualityext}

\noindent The aim of this Section is to construct an example to show that the results proven above for Poincar\'e duality groups over $F$ (again, $F$ is a field of characteristic $p$) can not be generalized to duality groups over $F$, thus finishing the proof of Theorem C. We are going to use Corollary 3.5 in \cite{jensenmeier} (or Theorem C in \cite{bradymeier}) which we state below. Recall first that an $n$-dimensional simplicial complex $L$ is called Cohen-Macaulay if for each simplex $\sigma\subset L$ (including the empty one), the reduced homology of the link of $\sigma$ in $L$ is concentrated in dimension $n-\dim\sigma-1$. 

\begin{theorem}[\cite{bradymeier}, \cite{jensenmeier}]\label{cohen} A right-angled Artin group $A_L$ is a duality group if and only if $L$ is Cohen-Macaulay.
\end{theorem}

\goodbreak
\begin{example}\label{notCM}{\rm Let $T$ be the 2-dimensional flag complex given by

\bigskip

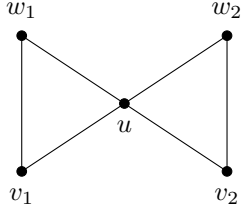
\begin{tikzpicture}[scale=0.9]

   \draw[black]
    (4,0) -- (4, 2) -- (7,0)--(7,2)--(4,0);
    
      \filldraw (5.5,1) circle (2pt) node[below=3pt]{$u$};
        \filldraw (4,0) circle (2pt) node[below=3pt]{$v_1$};
  \filldraw (4,2) circle (2pt) node[above=3pt]{$w_1$};
  \filldraw (7,0) circle (2pt) node[below=3pt]{$v_2$};
  \filldraw (7,2) circle (2pt) node[above=3pt]{$w_2$};

    \filldraw (1.5,1) circle (0pt) node[below=6pt]{$ $};

   \end{tikzpicture}
   
The reduced homology of the link of the simplex $\{u\}$ vanishes everywhere except at degree 0 where it is $\Z$. Therefore
$T$ is not Cohen-Macaulay.  Moreover, the link of all the other simplices is either contractible or, in the case of the 2-simplices, empty.
By the main theorem of \cite{jensenmeier}, if $\mathcal{X}(T)$ denotes the set of simplices of $T$:
$$\text{H}^*(A_{T},\Z A_T)=\bigoplus_{\sigma\in\mathcal{X}(T)}\Big[\bigoplus_{i=1}^\infty\overline{\text{H}}^{*-\dim\sigma-2}(\text{Lk}(\sigma))\Big]$$
Therefore:
$$\text{H}^2(A_{T},\Z A_T)=\bigoplus_{i=1}^\infty\Z\text{ and }$$
$$\text{H}^3(A_{T},\Z A_T)=\bigoplus_{2}\Big[\bigoplus_{i=1}^\infty\Z\Big].$$

By the universal coefficient Theorem for groups of type $\FP_\infty$ (\cite{Bieribook} Corollary 3.6), as the cohomology groups above are all torsion free we have
$$\text{H}^i(A_T,\Z A_T)\otimes_\Z F=\text{H}^i(A_T,FA_T)$$
thus $A_T$ is neither a duality group nor an $F$-duality group. }
\end{example}

\begin{example}\label{long}{\rm 

Let $W_n$ denote the Whitehead poset  (see \cite{bradymmm}  Definition 4.2).  
By Theorem 5.13 in \cite{bradymmm}, $W_n$ is Cohen-Macaulay.
Consider the case when $n=4$  and let the symmetric group on 4 letters, $S_4$, act on the poset $W_4$ by permuting the labeling. Let us denote by $L$ the simplicial complex associated to $W_4$, observe that $L$ is a flag complex with an admissible simplicial action of $S_4$. Let $K=\langle (1,2)\rangle\leq S_4$ be the transposition of $1$ and $2$ and put $\Gamma:=K\ltimes A_L$ with the obvious induced action of $K$ on $A_L$. By Theorem \ref{cohen}, $\Gamma$ is a virtual duality group, thus it is also $F$-duality for $F$ the field of two elements. One easily checks that the fixed points subcomplex $L^K$ is precisely the complex $T$ of Example \ref{notCM}.  
 Since $A_T=C_{A_L}(K)$ we deduce that $\Gamma$ is neither Bredon duality, nor $F$-Bredon duality.

}
\end{example}

The following construction is due to the referee:

\begin{example}\label{simpler}{\rm Consider any triangulation $S$ of the 2-sphere having at least one vertex, say $v$, of degree at least 6. Observe that there is a full subcomplex of $S$ isomorphic to the complex $T$ of Example \ref{notCM}. Let $L$ be union of two copies of $S$ glued along $T$.  
We claim that the complex $L$ is Cohen-Macaulay. To see observe that:
\begin{itemize} 
\item The link of any 0-simplex $\sigma$ has the homotopy type of  either a 1-sphere, a wedge sum of two 1-spheres (if $\sigma$ lies in $T\setminus\{v\}$) or a wedge sum of three 1-spheres (if $\sigma=\{v\}$).

\item The link of any 1-simplex $\sigma$ is either a set of two or a set of three points (if $\sigma$ lies in $T$).
\end{itemize}
Therefore for any simplex $\sigma$ the reduced homology of $\text{Lk}(\sigma)$ is concentrated at degree $2-\text{dim} \sigma-1$ so the claim follows. 
Now, let $K=C_2$ act in $L$ by swapping both copies of $S$. Obviously $L^K=T$. As before, $\Gamma$ is a virtual duality group but $\Gamma=K\ltimes A_L$ is neither Bredon duality, nor $F$-Bredon duality.

}\end{example}

\section{Acknowledgments}

This work was partially supported by Ministerio de Economia y Competitividad,  Secretar\'ia de Estado de Investigaci\'on, Desarrollo e Innovaci\'on BFM2010-19938-C03-03,  Grupo de Investigaci\'on Consolidado en \'Algebra del Gobierno de Arag\'on and  European Regional Development Funds. I would like to thank the anonymous referee for pointing out both a mistake in Example \ref{long} and the construction of  Example \ref{simpler}.

\bibliographystyle{abbrv}

\end{document}